\documentclass{article}
\usepackage[utf8]{inputenc}
\usepackage{amsthm}
\usepackage{amstext}
\usepackage{amsmath}
\usepackage{amssymb}
\usepackage[all]{xy}
\usepackage{color}
\usepackage{verbatim}

\usepackage{authblk}

\usepackage{enumitem}

\newtheorem{definition}{Def\text{}inition}[section]
\newtheorem{theorem}[definition]{Theorem}

\newtheorem{lemma}[definition]{Lemma}
\newtheorem{proposition}[definition]{Proposition}
\newtheorem{corollary}[definition]{Corollary}
\newtheorem{remark}[definition]{Remark}

\newtheorem{observation}[definition]{Observation}

\usepackage{tikz-cd}

\begin{document}

\title{ \bf \large Variations of star selection principles on small spaces}
\author{ \small JAVIER CASAS-DE LA ROSA AND SERGIO A. GARCIA-BALAN}
\date{}
\maketitle

\begin{abstract} 
In this paper, we introduce the notions of Star-$\sigma\mathcal{K}$ and absolutely Star-$\sigma\mathcal{K}$ spaces which allow us to unify results among several properties in the theory of star selection principles on small spaces. In particular, results on star selective versions of the Menger, Hurewicz and Rothberger properties and selective versions of property $(a)$ regarding the size of the space. Connections to other well-known star properties are mentioned. Furthermore, the absolute and selective version of the neighbourhood star selection principle are introduced. As an application, it is obtained that the extent of a separable absolutely strongly star-Menger (absolutely strongly star-Hurewicz) space is at most the dominating number $\mathfrak{d}$ (the bounding number $\mathfrak{b}$). 
\end{abstract}

\let\thefootnote\relax\footnote{\today }
\let\thefootnote\relax
\let\thefootnote\relax


\section{Introduction}
\label{intro}
In this section, we recall some classic definitions and important results in the theory of star selection principles that are central to our work. In addition, we introduce useful notation and terminology that will help us to deal with variations of the classical star versions of Menger, Hurewicz and Rothberger properties and selective variations property $(a)$. Main results, consequences and applications are in Section 2,  and 3. In Section 4 we introduce new variations of neighbourhood star selection principles.

\subsection{Notation and terminology}

Let $X$ be a topological space. We denote by $[X]^{<\omega}$ the collection of all finite subsets of $X$. 
For a subset $A$ of $X$ and a collection $\mathcal{U}$ of subsets of $X$, the star of $A$ with respect to $\mathcal{U}$, denoted by $St(A,\mathcal{U})$, is the set $\bigcup\{U\in\mathcal{U}:U\cap A\neq\emptyset\}$; for $A=\{x\}$ with $x\in X$, we write $St(x,\mathcal{U})$ instead of $St(\{x\},\mathcal{U})$. Throughout this paper, all spaces are assumed to be regular, unless a specific separation axiom is indicated. For notation and terminology, we refer to \cite{E}.

In the context of classical star covering properties, we follow the notation of \cite{DRRT}. Recall that a space $X$ is said to be strongly starcompact (strongly star-Lindel\"{o}f), briefly $SSC$ ($SSL$), if for every open cover $\mathcal{U}$ of $X$ there exists a finite (countable) subset $F$ of $X$ such that $St(F,\mathcal{U})=X$. A space $X$ is starcompact (star-Lindel\"{o}f), briefly $SC$ ($SL$), if for every open cover $\mathcal{U}$ of $X$ there exists a finite (countable) subset $\mathcal{V}$ of $\mathcal{U}$ such that $St(\bigcup\mathcal{V},\mathcal{U})=X$. It is worth to mention that countable compactness and strongly starcompactness are equivalent for Hausdorff spaces (see \cite{DRRT}). We refer the reader to the survey of Matveev \cite{M} for a more detailed treatment of these star covering properties.

\subsection{Classical (star) selection principles}\label{classicalSSP}

We recall the definition of three classical well-known selection principles and its star versions. Let $\mathcal{A}$ and $\mathcal{B}$ be families of sets. The following general forms of classical selection principles were introduced by M. Scheepers in \cite{MS1}:\\
\newline
$\mathbf{S_1(\mathcal{A},\mathcal{B})}$: For each sequence $\{A_n:n\in\omega\}$ of elements of $\mathcal{A}$ there is a sequence $\{b_n:n\in\omega\}$ such that for each $n\in\omega$, $b_n\in A_n$ and $\{b_n:n\in\omega\}$ is an element of $\mathcal{B}$.\\
\newline
$\mathbf{S_{fin}(\mathcal{A},\mathcal{B})}$: For each sequence $\{A_n:n\in\omega\}$ of elements of $\mathcal{A}$ there is a sequence $\{B_n:n\in\omega\}$ such that for each $n\in\omega$, $B_n\in [A_n]^{<\omega}$ and $\bigcup\{B_n:n\in\omega\}$ is an element of $\mathcal{B}$.\\
\newline
$\mathbf{U_{fin}(\mathcal{A},\mathcal{B})}$: For each sequence $\{A_n:n\in\omega\}$ of elements of $\mathcal{A}$ there is a sequence $\{B_n:n\in\omega\}$ such that for each $n\in\omega$, $B_n\in [A_n]^{<\omega}$ and $\{\bigcup B_n:n\in\omega\}$ is an element of $\mathcal{B}$.\\

Given a topological space $X$, we denote by $\mathcal{O}$ the collection of all open covers of $X$ and by $\Gamma$ the collection of all $\gamma$-covers of $X$; an open cover $\mathcal{U}$ of $X$ is a $\gamma$-cover if it is infinite and each $x\in X$ belongs to all but finitely many elements of $\mathcal{U}$. Thus,
\begin{description}
\item[M]: $S_{fin}(\mathcal{O},\mathcal{O})$ defines the classical Menger covering property (see \cite{MEN});
\item[R]: $S_1(\mathcal{O},\mathcal{O})$ defines the classical Rothberger covering property (see \cite{R});
\item[H]: $U_{fin}(\mathcal{O},\Gamma)$ defines the classical Hurewicz covering property (see \cite{H}).
\end{description}

The following star selection principles were introduced in \cite{K}. Henceforth, $\mathcal{A}$ and $\mathcal{B}$ will denote some collections of open covers of a space $X$ and $\mathcal{K}$ a family of subsets of $X$:\\
\newline
$\mathbf{S^*_1(\mathcal{A},\mathcal{B})}$: For each sequence $\{\mathcal{U}_n:n\in\omega\}\subseteq\mathcal{A}$ there exists a sequence $\{U_n:n\in\omega\}$ such that $U_n\in\mathcal{U}_n$, $n\in\omega$, and $\{St(U_n,\mathcal{U}_n):n\in\omega\}\in\mathcal{B}$.\\
\newline
$\mathbf{S^*_{fin}(\mathcal{A},\mathcal{B})}$: For each sequence $\{\mathcal{U}_n:n\in\omega\}\subseteq\mathcal{A}$ there exists a sequence $\{\mathcal{V}_n:n\in\omega\}$ such that $\mathcal{V}_n$ is a finite subset of $\mathcal{U}_n$, $n\in\omega$, and $\{St(\bigcup\mathcal{V}_n,\mathcal{U}_n:n\in\omega\}\in\mathcal{B}$.\\
\newline
$\mathbf{SS^*_{\mathcal{K}}(\mathcal{A},\mathcal{B})}$: For each sequence $\{\mathcal{U}_n:n\in\omega\}\subseteq\mathcal{A}$ there exists a sequence $\{K_n:n\in\omega\}\subseteq\mathcal{K}$ such that $\{St(K_n,\mathcal{U}_n):n\in\omega\}\in\mathcal{B}$.\\

When $\mathcal{K}$ is the collection of all finite (resp. one-point) subsets of $X$, it is denoted by $\mathbf{SS^*_{fin}(\mathcal{A},\mathcal{B})}$ (resp. $\mathbf{SS^*_1(\mathcal{A},\mathcal{B})}$) instead of $\mathbf{SS^*_{\mathcal{K}}(\mathcal{A},\mathcal{B})}$. Following this terminology, the star versions for the cases Menger and Rothberger were defined in \cite{K} and the star versions for the Hurewicz case were defined in \cite{BCK}:

\begin{description}
\item[SM]: $S^*_{fin}(\mathcal{O},\mathcal{O})$ defines the star-Menger property (see \cite{K});
\item[SSM]: $SS^*_{fin}(\mathcal{O},\mathcal{O})$ defines the strongly star-Menger property (see \cite{K});
\item[SR]: $S^*_1(\mathcal{O},\mathcal{O})$ defines the star-Rothberger property (see \cite{K});
\item[SSR]: $SS^*_1(\mathcal{O},\mathcal{O})$ defines the strongly star-Rothberger property (see \cite{K});
\item[SH]: $S^*_{fin}(\mathcal{O},\Gamma)$ defines the star-Hurewicz property (see \cite{BCK});
\item[SSH]: $SS^*_{fin}(\mathcal{O},\Gamma)$ defines the strongly star-Hurewicz property (see \cite{BCK}).
\end{description}

For paracompact Hausdorff spaces the three Menger-type properties, $SM$, $SSM$ and $M$ are equivalent and the same situation holds for the three Rothberger-type properties and the three Hurewicz-type properties (see \cite{K} and \cite{BCK}). In fact, the previous equivalences also holds in paraLindel\"of spaces (see \cite{CGS}).\\

The following diagram shows the relationships among these properties (in the diagram $C$ and $L$ are used to denote compactness and the Lindel\"{o}f property, respectively). None of the arrows in the following diagram reverse. We refer the reader to \cite{K_survey} to see the current state of knowledge about these relationships with others.

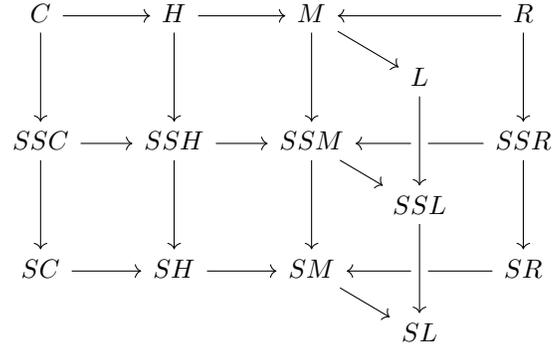
\begin{figure}[h!]
\[
\begin{tikzcd}[row sep=1em, column sep = 1em]
C \arrow[rr] \arrow[dd,swap] && H \arrow[rr] \arrow[dd] && M  \arrow[rr,<-]  \arrow[dr] \arrow[dd] &&
  R \arrow[dd ] \\
&& && & L &&  \\
SSC \arrow[rr] \arrow[dd] && SSH \arrow[rr] \arrow[dd] && SSM \arrow[rr,<-] \arrow[dr]  \arrow[dd]  && SSR \arrow[dd]  \\
&& && & SSL  \arrow[uu,<-,crossing over]&&  \\
SC \arrow[rr] && SH \arrow[rr] && SM \arrow[rr,<-] \arrow[dr] && SR \\
&& && & SL  \arrow[uu,<-,crossing over]&&  
\end{tikzcd}
\]
\end{figure}

\newpage

The first important characterizations of some of these star selection principles for the case of $\Psi$-spaces were obtained by Bonanzinga and Matveev:

\begin{proposition}[\cite{BM}]
\label{BMCharacterization}
Given any almost disjoint family $\mathcal{A}$, the following assertions hold.
\begin{enumerate}
\item $\Psi(\mathcal{A})$ is strongly star-Menger if and only if $|\mathcal{A}| < \mathfrak{d}$.
\item $\Psi(\mathcal{A})$ is strongly star-Hurewicz if and only if $|\mathcal{A}| < \mathfrak{b}$.
\item If $|\mathcal{A}| < cov(\mathcal{M})$, then $\Psi(\mathcal{A})$ is strongly star-Rothberger.
\end{enumerate}
\end{proposition}


The following generalization of one direction of (1) in the previous result was given by Sakai in \cite{SM}. It can be viewed as a selective version of the fact that Lindel\"of spaces of size less than $\mathfrak{d}$ are Menger:

\begin{proposition}\cite{SM}
\label{SSL+dimpliesSSM}
Every strongly star-Lindel\"{o}f space of cardinality less than $\mathfrak{d}$ is strongly star-Menger.
\end{proposition}

The Hurewicz and Rothberger cases of Proposition \ref{SSL+dimpliesSSM} can be proved using similar ideas:

\begin{proposition}\cite{CGS}
\label{SSL+bimpliesSSH}
Every strongly star-Lindel\"{o}f space of cardinality less than $\mathfrak{b}$ is strongly star-Hurewicz.
\end{proposition}

\begin{proposition}
\label{SSL+covMimpliesSSR}
Every strongly star-Lindel\"{o}f space of cardinality less than $cov(\mathcal{M})$ is strongly star-Rothberger.
\end{proposition}


Other generalizations of (1) and (2) of Proposition \ref{BMCharacterization} that also characterizes star selection principles on the Niemytzki plane were given in \cite{CGS}: 

\begin{theorem}\cite{CGS}
\label{CharacterizationofSSM}
Let $X$ be a topological space of the form $Y\cup Z$, where $Y\cap Z=\emptyset$, $Z$ is a $\sigma$-compact subspace and $Y$ is a closed discrete set. If $X$ is strongly star-Lindel\"{o}f, then $|Y|<\mathfrak{d}$ if and only if $X$ is strongly star-Menger.
\end{theorem}

\begin{theorem}\cite{CGS}
\label{CharacterizationofSSH}
Let $X$ be a topological space of the form $Y\cup Z$, where $Y\cap Z=\emptyset$, $Z$ is a $\sigma$-compact subspace and $Y$ is a closed discrete set. If $X$ is strongly star-Lindel\"{o}f, then $|Y|<\mathfrak{b}$ if and only if $X$ is strongly star-Hurewicz.
\end{theorem}

\subsection{Absolute and selective versions of star selection principles}
\label{ASelstarSL}

In this section we recall the absolute and selective versions of the classical star selection principles. We start by mentioning the definition of the absolute and selective versions of the strongly star-Lindel\"of property. In \cite{Bo}, Bonanzinga defined and studied the absolute version of the strongly star-Lindel\"of property.

\begin{definition}\cite{Bo}
A space $X$ is absolutely strongly star-Lindel\"{o}f ($aSSL$) if for any open cover $\mathcal{U}$ of $X$ and any dense subset $D$ of $X$, there is a countable set $C\subseteq D$ such that $St(C,\mathcal{U})=X$.
\end{definition}

On the other hand, the selective version of the strongly star-Lindel\"of property was defined first by S. Bhowmik in \cite{B} and later studied in \cite{BCS} with a different name.

\begin{definition}\cite{B}
A space $X$ is selectively strongly star-Lindel\"of ($selSSL$) if for every open cover $\mathcal{U}$ of $X$ and for every sequence $\{D_n:n\in\omega\}$ of dense sets of $X$, there is a sequence $\{F_n:n\in\omega\}$ of finite sets such that $F_n\subseteq D_n$, $n\in\omega$, and $\{St(F_n,\mathcal{U}):n\in\omega\}$ is an open cover of $X$.
\end{definition}

Caserta, Di Maio and Ko\v{c}inac introduced in \cite{CDK} the absolute versions of the classical star selection principles in a general form. Here, we will use a different notation that seems to be simpler and naturally relates to the star selection principle $SS^*_{\mathcal{K}}(\mathcal{A},\mathcal{B})$:$^1$ \footnote{$^1$In \cite{CDK}, the authors employed an idea of Matveev to define, in a different general form, the absolute versions of star selection principles. Since part of the motivation for that general form was to give the selective version of the property $(a)$, the word \emph{selectively} was used as part of that terminology (see for instance Section 5 in \cite{K_survey}). Here we prefer to use a different terminology because of the introduction of the selective versions of star selection principles.}

\begin{definition}
Given a space $X$, the following selection hypothesis is defined:\\
\newline
$\mathbf{absolutely \; SS^*_{\mathcal{K}}(\mathcal{A},\mathcal{B})}$: For each sequence $\{\mathcal{U}_n:n\in\omega\}\subseteq\mathcal{A}$ and each dense subset $D$ of $X$, there exists a sequence $\{K_n:n\in\omega\}\subseteq\mathcal{K}$ such that each $K_n\subseteq D$, $n\in\omega$, and $\{St(K_n,\mathcal{U}_n):n\in\omega\}\in\mathcal{B}$. 
\end{definition}

For shortness, we write $\mathbf{aSS^*_{\mathcal{K}}(\mathcal{A},\mathcal{B})}$ instead of $\mathbf{absolutely \; SS^*_{\mathcal{K}}(\mathcal{A},\mathcal{B})}$ and, as usual, when $\mathcal{K}$ is the collection of all finite (resp. one-point, closed discrete) subsets of $X$, we write $\mathbf{aSS^*_{fin}(\mathcal{A},\mathcal{B})}$ (resp. $\mathbf{aSS^*_1(\mathcal{A},\mathcal{B})}$, $\mathbf{aSS^*_{cd}(\mathcal{A},\mathcal{B})}$) instead of $\mathbf{aSS^*_{\mathcal{K}}(\mathcal{A},\mathcal{B})}$. Following this terminology, the absolute versions of the classical star selection principles (defined in \cite{CDK}) are given as follows:

\begin{description}
\item[aSSM]: $aSS^*_{fin}(\mathcal{O},\mathcal{O})$ defines the absolutely strongly star-Menger property;
\item[aSSR]: $aSS^*_1(\mathcal{O},\mathcal{O})$ defines the absolutely strongly star-Rothberger property;
\item[aSSH]: $aSS^*_{fin}(\mathcal{O},\Gamma)$ defines the absolutely strongly star-Hurewicz property.
\end{description}

More recently, Bonanzinga et al. defined and studied the selective version of the strongly star-Menger property in \cite{BoMa} and \cite{Cuz}. Furthermore, they asked the question whether absolutely strongly star-Menger and selectively strongly star-Menger are equivalent properties. This selective principle naturally gives birth to the selective version for the Hurewicz and Rothberger cases. We introduce the following general notation that includes these kind of selection principles.

\begin{definition}
Given a space $X$, the following selection hypothesis is defined:\\
\newline
$\mathbf{selectively \; SS^*_{\mathcal{K}}(\mathcal{A},\mathcal{B})}$: For each sequence $\{\mathcal{U}_n:n\in\omega\}\subseteq\mathcal{A}$ and each sequence $\{D_n:n\in\omega\}$ of dense sets of $X$, there exists a sequence $\{K_n:n\in\omega\}\subseteq\mathcal{K}$ such that each $K_n\subseteq D_n$, $n\in\omega$, and $\{St(K_n,\mathcal{U}_n):n\in\omega\}\in\mathcal{B}$. 
\end{definition}

For shortness, we write $\mathbf{selSS^*_{\mathcal{K}}(\mathcal{A},\mathcal{B})}$ instead of $\mathbf{selectively \; SS^*_{\mathcal{K}}(\mathcal{A},\mathcal{B})}$ and, again, when $\mathcal{K}$ is the collection of all finite (resp. one-point, closed discrete) subsets of $X$, we write $\mathbf{selSS^*_{fin}(\mathcal{A},\mathcal{B})}$ (resp. $\mathbf{selSS^*_1(\mathcal{A},\mathcal{B})}$, $\mathbf{selSS^*_{cd}(\mathcal{A},\mathcal{B})}$) instead of $\mathbf{selSS^*_{\mathcal{K}}(\mathcal{A},\mathcal{B})}$. With this notation, the selective versions of some classical star selection principles are given as follow:

\begin{description}
\item[selSSM]\cite{BoMa} \cite{Cuz}: $selSS^*_{fin}(\mathcal{O},\mathcal{O})$ defines the selectively strongly star-Menger property;
\item[selSSR]: $selSS^*_1(\mathcal{O},\mathcal{O})$ defines the selectively strongly star-Rothberger property;
\item[selSSH]: $selSS^*_{fin}(\mathcal{O},\Gamma)$ defines the selectively strongly star-Hurewicz property.
\end{description}

The natural relationships among the absolute and selective versions are given in the following diagram:

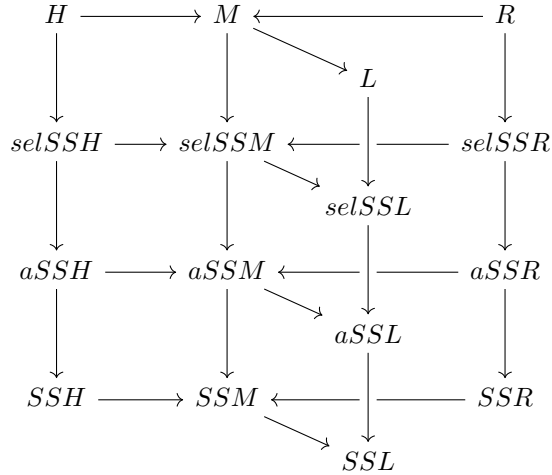
\begin{figure}[h!]
\[
\begin{tikzcd}[row sep=1em, column sep = 1em]
H \arrow[rr] \arrow[dd] && M  \arrow[rr,<-]  \arrow[dr] \arrow[dd] &&
  R \arrow[dd ] \\
&& & L &&  \\
selSSH \arrow[rr] \arrow[dd] && selSSM  \arrow[rr,<-]  \arrow[dr] \arrow[dd] &&
  selSSR \arrow[dd ] \\
&& & selSSL \arrow[uu,<-,crossing over] &&  \\
aSSH \arrow[rr] \arrow[dd] && aSSM \arrow[rr,<-] \arrow[dr]  \arrow[dd]  && aSSR \arrow[dd]  \\
&& & aSSL  \arrow[uu,<-,crossing over]&&  \\
SSH \arrow[rr] && SSM \arrow[rr,<-] \arrow[dr]  && SSR  \\
&& & SSL  \arrow[uu,<-,crossing over]&&  
\end{tikzcd}
\]
\caption{Absolute and Selective Versions} 
\end{figure}

Analogous to the question posed by Bonanzinga et al. in \cite{BoMa} and \cite{Cuz}, whether there is an absolutely strongly star-Menger not selectively strongly star-Menger space, it is important to figure out whether the absolute and selective versions of the Rothberger and Hurewicz cases are equivalent, respectively.\\

It is natural to wonder whether similar results as Proposition \ref{SSL+dimpliesSSM}, \ref{SSL+bimpliesSSH} and \ref{SSL+covMimpliesSSR} also hold in the absolute and selective context. We provide an affirmative answer in a broader sense in Section \ref{starSL}.

\subsection{Some Selective versions of Property $(a)$}
\label{SelectiveVerProp(a)}

In \cite{DRRT}, van Douwen et.al. showed that for Hausdorff spaces, countable compactness is equivalent to strongly starcompactness. Motivated by this equivalence, Matveev defined in \cite{Macc} the absolute version of the strongly starcompact property: 

\begin{definition}\cite{Macc}
A space $X$ is absolutely countably compact ($acc$) if for any open cover $\mathcal{U}$ of $X$ and any dense subset $D$ of $X$, there is a finite set $F\subseteq D$ such that $St(F,\mathcal{U})=X$.
\end{definition}
Later, using a similar idea, the following interesting property was also introduced by Matveev in \cite{Mpropertya}: 

\begin{definition}\cite{Mpropertya}
A space $X$ has property (a) if for every open cover $\mathcal{U}$ of $X$ and each dense set $D$ of $X$, there exists $C \subseteq D$ closed and discrete subset of $X$ such that $St(C,\mathcal{U}) = X$.
\end{definition}
Then, a selective version of property $(a)$, called \emph{selectively $(a)$}, was given by Caserta, Di Maio and Ko\v{c}inac in \cite{CDK} using a general selection hypothesis. We mention the selectively $(a)$ property in terms of our notation:

\begin{definition}\cite{CDK}
A space $X$ is selectively (a) if satisfies $aSS^*_{cd}(\mathcal{O},\mathcal{O})$.
\end{definition}

In \cite{CDK}, the authors pointed out that if $X$ is a separable selectively $(a)$ space, then every closed discrete subset of $X$ is of size less than $\mathfrak{c}$. The general case was proved by da Silva in  \cite{dSSG2014} where he investigated the selectively $(a)$ property in $\Psi$-spaces. In particular, he obtained the following result:

\begin{theorem}
\label{DasilvaA}
\cite{dSSG2014}
Let $\mathcal{A}$ be an infinite almost disjoint family on $\omega$. Then
\begin{enumerate}
    \item If the size of $\mathcal{A}$ is less than $\mathfrak{d}$, then $\Psi(\mathcal{A})$ is selectively $(a)$.
    \item Assume $\mathcal{A}$ is maximal. Then $\Psi(\mathcal{A})$ is selectively $(a)$ if and only if $|\mathcal{A}|<\mathfrak{d}$.
\end{enumerate}
\end{theorem}

Variations of the previous notion will play an important role in sections 2 and 3. In particular, a stronger variation of the selectively $(a)$ property which is mentioned in \cite{PSS}  by Passos, Santana, and da Silva (called $\dagger$), and also introduced in \cite{Cuz}  with the following name:

\begin{definition}
A space $X$ is strongly selectively $(a)$ if satisfies $selSS^*_{cd}(\mathcal{O},\mathcal{O})$.
\end{definition}

It is worth mentioning that this property is currently being studied by Bonanzinga and Maesano in \cite{BoMa}.


\newpage

\section{Star selection principles in small spaces}
\label{starSL}

It is well-known that Lindel\"of spaces of size less than $\mathfrak{d}$ are Menger and Lindel\"of spaces of size less than $\mathfrak{b}$ are Hurewicz. Some star versions of these kind of results were presented in subsections \ref{classicalSSP} and \ref{SelectiveVerProp(a)}. The goal of this section is to present a new way to bring all these results together.


\subsection{General theorems on small spaces}

The notion of \emph{star-$\mathbb{P}$} was introduced in \cite{MTW} (see also \cite{M}) and its absolute version, namely, \emph{absolutely star-$\mathbb{P}$}  in \cite{YKS00}. Paying attention to the following concepts turned out to be essential in our main results:

\begin{definition}
\label{starP}
Given $\mathcal{K}$ a family of subsets of a space $X$, we call $X$:
\begin{itemize}[nosep]
\item \emph{Star-$\sigma\mathcal{K}$}  if for each open cover $\mathcal{U}$ of $X$ there is $K \subseteq X$ so that $K$ is a $\sigma\mathcal{K}$ kernel of $X$ with respect to $\mathcal{U}$. That is, $K$ is a countable union of subsets of $X$ each one belonging to $\mathcal{K}$ and $St(K,\mathcal{U}) = X$.
\item \emph{absolutely Star-$\sigma\mathcal{K}$}, (abbreviated aStar-$\sigma\mathcal{K}$) if for each $D$ dense subset of $X$ and for each open cover $\mathcal{U}$ of $X$ there is $K \subseteq D$ so that $K$ is a $\sigma\mathcal{K}$ kernel of $X$ with respect to $\mathcal{U}$. That is, $K$ is a countable union of subsets of $D$ each one belonging to $\mathcal{K}$ and $St(K,\mathcal{U}) = X$.
\end{itemize}
\end{definition}

\begin{observation}
\label{relationship}
If $\mathcal{B}$ is either $\mathcal{O}$ or $\Gamma$ and $\mathcal{K}$ is a family of subsets of a space $X$, then we have

\begin{center}
    $selSS^*_\mathcal{K}(\mathcal{O},\mathcal{B}) \rightarrow $ $aSS^*_\mathcal{K}(\mathcal{O},\mathcal{B}) \rightarrow  $ aStar-$\sigma\mathcal{K} \rightarrow   $ Star-$\sigma\mathcal{K}$.
\end{center}

\end{observation}

\begin{theorem}
\label{general01}
Given a space $X$ of size less than $\mathfrak{d}$ and $\mathcal{K}$ a family of subsets of $X$ which is closed under finite unions, then
\begin{enumerate}[nosep]
\item if $X$ is star-$\sigma\mathcal{K}$, then $X$ is $SS^*_\mathcal{K}(\mathcal{O},\mathcal{O})$.
\item if $X$ is aStar-$\sigma\mathcal{K}$, then $X$ is $selSS^*_\mathcal{K}(\mathcal{O},\mathcal{O})$.
\end{enumerate}
\end{theorem}

\begin{proof}
We will prove item 2 as item 1 follows similarly. Hence, assume $X$ is absolutely star-$\sigma\mathcal{K}$ and $|X| < \mathfrak{d}$. Let $\{ \mathcal{U}_n : n \in \omega \}$ be a sequence of open covers of $X$  and let $\{ D_n : n \in \omega \}$ be any sequence of dense subsets of $X$. For $n \in \omega$, let $E_n \subseteq D_n$ be a $\sigma\mathcal{K}$ subset of $X$ so that $St(E_n,\mathcal{U}_n) = X$. Thus, for each $n \in \omega$, $E_n = \bigcup _{m\in\omega}E_n^m$ where for each $n,m \in \omega$, $E_n^m\in\mathcal{K}$. Let us list $X$ as $\{x_\alpha:\alpha < \kappa\}$ with $\kappa < \mathfrak{d}$. For each $n\in \omega$ and each $\alpha < \kappa$ let $f_\alpha(n) =min\{m\in\omega:x_\alpha \in St(E_n^m,\mathcal{U}_n)\}$. Since the collection $\{f_\alpha: \alpha < \kappa \}$ has size less than $\mathfrak{d}$, there exists $g \in \omega^\omega$ such that for every $\alpha < \kappa$, $g \not \le ^* f_\alpha$. 
For each $n \in \omega$, let $C_n =\bigcup _{m\leq g(n)} E_n^m$. For each $n\in \omega$, since $C_n \subseteq D_n$ is a finite union of elements of $\mathcal{K}$, and $\mathcal{K}$ is closed under finite unions, $C_n\in \mathcal{K}$ and $\{St(C_n,\mathcal{U}_n):n\in \omega\}$ is an open cover of $X$. Hence,  $X$ is $selSS^*_\mathcal{K}(\mathcal{O},\mathcal{O})$.
\end{proof}

With similar ideas it also holds true:

\begin{theorem}
\label{generalHurewicz}
Given a space $X$ of size less than $\mathfrak{b}$ and $\mathcal{K}$ a family of subsets of $X$ which is closed under finite unions, then
\begin{enumerate}[nosep]
\item if $X$ is star-$\sigma\mathcal{K}$, then $X$ is $SS^*_\mathcal{K}(\mathcal{O},\Gamma)$.
\item if $X$ is aStar-$\sigma\mathcal{K}$, then $X$ is $selSS^*_\mathcal{K}(\mathcal{O},\Gamma)$.
\end{enumerate}
\end{theorem}

The following diagram summarizes the results presented in Theorem \ref{general01} and Theorem \ref{generalHurewicz}:

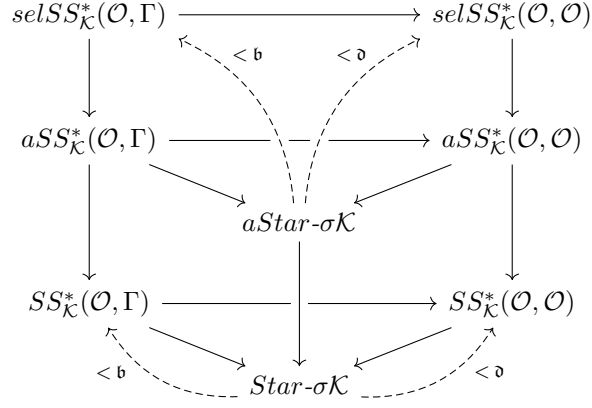
\begin{figure}[h!]
\[
\begin{tikzcd}[row sep=1.5em, column sep = 2em]
sel SS^*_\mathcal{K}(\mathcal{O},\Gamma) \arrow[dd]  \arrow[rr] & & selSS^*_\mathcal{K}(\mathcal{O},\mathcal{O})  \arrow[dd]    \\
& & \\
aSS^*_\mathcal{K}(\mathcal{O},\Gamma) \arrow[rr]  \arrow[dr]  \arrow[dd] & & aSS^*_\mathcal{K}(\mathcal{O},\mathcal{O}) \arrow[dd]  \arrow[dl]   \\
& aStar\textit{-}\sigma\mathcal{K}    \arrow[uuul, crossing over,swap, "< \, \mathfrak{b}" near end, bend right, dashed]  \arrow[uuur,  crossing over,  "< \, \mathfrak{d}" near end, bend left, dashed] & \\
SS^*_\mathcal{K}(\mathcal{O},\Gamma) \arrow[rr] \arrow[dr] & & SS^*_\mathcal{K}(\mathcal{O},\mathcal{O})       \arrow[dl]  \\
& Star\textit{-}\sigma\mathcal{K}  \arrow[uu,  <- ,crossing over] \arrow[ur , swap, "< \, \mathfrak{d}", near end,bend right, dashed] \arrow[ul , "< \, \mathfrak{b}" near end,bend left, dashed]  & 
\end{tikzcd}
\]
\caption{$\mathcal{K}$ is closed under finite unions.} 
\end{figure}

Observe that the family of ``one-point" subsets of a space $X$ is not closed under finite unions. Therefore, an analogous result as Theorems \ref{general01} and \ref{generalHurewicz}, for the Rothberger case is not possible. Nevertheless, the following holds:

\begin{proposition}
\label{aSSLSMALLcovMimpliesSelSSM}
For any space $X$, if $X$ is absolutely strongly star-Lindel\"of and it has size less than $cov(\mathcal{M})$, then $X$ is selectively strongly star-Rothberger.
\end{proposition}

\begin{proof}
Assume $X = \{x_\alpha: \alpha < \kappa\}$ is $aSSL$ such that $\kappa < cov(\mathcal{M})$. Let $(\mathcal{U}_n: n\in \omega)$ be any sequence of open covers and let $(D_n:n\in \omega)$ be any sequence of dense subsets of $X$. For each $n \in \omega$, fix  $E_n =\{d_q^n:q\in \omega\}\in [D_n]^{\leq \omega}$ such that $St(E_n,\mathcal{U}_n) = X$.\\
For each $\alpha < \kappa$ define $f_\alpha\in \omega^\omega$ such that for $n\in\omega$, $f_\alpha(n)=min\{m\in\omega:x_\alpha \in St(d^n_m,\mathcal{U}_n)\}$. Since $|\{f_\alpha:\alpha < \kappa\}| < cov(\mathcal{M})$, there exists $g \in \omega^\omega$ such that $g$ guesses $\{f_\alpha:\alpha < \kappa\}$, i.e. for each $\alpha < \kappa$, $\{n\in \omega : g(n) = f_\alpha(n)\}$ is infinite. It follows that $\{St(d_n^{g(n)},\mathcal{U}_n):n\in \omega\}$ is an open cover of $X$.
\end{proof}

\begin{remark}
It would be interesting to investigate what happens when we change $\mathcal{O}$ or $\Gamma$ in Theorems \ref{general01} and \ref{generalHurewicz} by other kind of subcollections of open covers.
\end{remark}

\subsection{Consequences of General Theorems}\label{corollaries}

Here, we collect some immediate consequences of Theorems \ref{general01} and \ref{generalHurewicz} proved in the previous subsection.

\begin{corollary}
\label{trivilind}
Let $X$ be any space of size less than $\mathfrak{d}$. Then
\begin{enumerate}[nosep]
\item \cite{SM} If $X$ is $SSL$, then $X$ is $SS^*_{fin}(\mathcal{O},\mathcal{O})$ (i.e. $SSM$).
\item If $X$ is $aSSL$, then $X$ is $selSS^*_{fin}(\mathcal{O},\mathcal{O})$ (i.e. $SelSSM$).$^2$
\item If $X$ is aStar-$\sigma$-$cd$, then $X$ is $selSS^*_{cd}(\mathcal{O},\mathcal{O})$ (i.e. strongly selectively $(a)$).
\end{enumerate}
\end{corollary}

\begin{proof}
To prove (1) and (2), let $\mathcal{K} = fin$ be the family of all finite subsets of $X$ and observe that $X$ is $SSL$ if and only if it is star-$\sigma\mathcal{K}$ and $X$ is $aSSL$ if and only if it is aStar-$\sigma\mathcal{K}$. Since $fin$ is closed under finite unions, apply Theorem \ref{general01} to get the result. To prove (3) let $\mathcal{K} = cd$ be the family of all closed discrete subsets of $X$ and apply Theorem \ref{general01}.
\end{proof}

Observe that Corollary  \ref{trivilind} (1) is Sakai's Proposition \ref{SSL+dimpliesSSM}. Corollary \ref{trivilind} (2) generalizes Proposition 9 in \cite{BCS} where Bonanzinga, Cuzzupe and Sakai show that $aSSL$ spaces of size less than $\mathfrak{d}$ are $selSSL$. Given that $\Psi$-spaces are $aStar$-$\sigma$-$cd$, Corollary \ref{trivilind} (3) generalizes da Silva's Theorem \ref{DasilvaA} (1).\\

Similarly to Corollary \ref{trivilind}, we have the following for the Hurewicz case:

\begin{corollary}
\label{trivilinb}
Let $X$ be any space of size less than $\mathfrak{b}$. Then
\begin{enumerate}[nosep]
\item If $X$ is $SSL$, then $X$ is $SS^*_{fin}(\mathcal{O},\Gamma)$ (i.e. $SSH$).
\item If $X$ is $aSSL$, then $X$ is $selSS^*_{fin}(\mathcal{O},\Gamma)$ (i.e. $SelSSH$).
\item If $X$ is aStar-$\sigma$-$cd$, then $X$ is $selSS^*_{cd}(\mathcal{O},\Gamma)$.
\end{enumerate}
\end{corollary}

Observe that Corollary  \ref{trivilinb} (1) is precisely Proposition \ref{SSL+bimpliesSSH}. Corollary \ref{trivilind} (2) and Corollary \ref{trivilinb} (2) improve Song's Remark 2.5 in \cite{YKS2013}, and Remark 2.6 (see also Remark 2.4) in \cite{YKS02}: in $\Psi$-spaces the properties strongly star-Menger and absolutely strongly star-Menger are equivalent and the properties strongly star-Hurewicz and absolutely strongly star-Hurewicz are equivalent. Furthermore, they allow us to provide the following results that can be seen as partial analogous of Theorems \ref{CharacterizationofSSM} and \ref{CharacterizationofSSH} for absolutely strongly star Lindel\"of spaces:

\begin{proposition}
\label{CharacterizationofaSSMYUZ}
Let $X$ be a topological space of the form $Y\cup Z$, where $Y\cap Z=\emptyset$, $Z$ is a $\sigma$-compact subspace and $Y$ is a closed discrete set. If $X$ is absolutely strongly star-Lindel\"{o}f and $|Y|<\mathfrak{d}$, then $X$ is selectively strongly star-Menger.
\end{proposition}

\begin{proposition}
\label{CharacterizationofaSSHYUZ}
Let $X$ be a topological space of the form $Y\cup Z$, where $Y\cap Z=\emptyset$, $Z$ is a $\sigma$-compact subspace and $Y$ is a closed discrete set. If $X$ is absolutely strongly star-Lindel\"{o}f and $|Y|<\mathfrak{b}$, then $X$ is selectively strongly star-Hurewicz.
\end{proposition}


\begin{observation}
For $\Psi$-spaces and for the Niemytzki plane, the three properties SSM, aSSM and selSSM are equivalent, and the three properties SSH, aSSH and selSSH are equivalent.$^2$ \footnote{$^2$In private communication with M. Bonanzinga, she informed to the authors that, together with F. Maesano in \cite{BoMa}, they obtained in a direct way a proof of Corollary \ref{trivilind} (2) and the equivalences of $SSM$, $aSSM$ and $selSSM$ for $\Psi$-spaces.}
\end{observation}

\newpage

The following diagram sum things up for the Menger, Rothberger and Hurewicz cases:

\begin{figure}[h!]
\[
\begin{tikzcd}[row sep=2em, column sep = 2.5em]
selSSH \arrow[rr] \arrow[dd] && selSSM  \arrow[rr,<-]  \arrow[dr] \arrow[dd] &&
  selSSR \arrow[dd ] \\
&& & selSSL   &&  \\
aSSH \arrow[rr] \arrow[dd] && aSSM \arrow[rr,<-] \arrow[dr]  \arrow[dd]  && aSSR \arrow[dd]  \\
&& & aSSL  \arrow[uuulll, crossing over, bend left, swap, "< \, \mathfrak{b}" near end, dashed] \arrow[uu,<-,crossing over]  \arrow[uuul, crossing over, "< \, \mathfrak{d}" near end, dashed] \arrow[uuur, crossing over,swap, "< \, cov(\mathcal{M})" near start, dashed]&&  \\
SSH \arrow[rr] && SSM \arrow[rr,<-] \arrow[dr]  && SSR \arrow[dl]   \\
&& & SSL \arrow[uu,<-,crossing over] \arrow[ur , swap, "< \, cov(\mathcal{M})", near end,bend right, dashed] \arrow[ul , "< \, \mathfrak{d}" near end,bend left, dashed] \arrow[ulll, bend left, swap, "< \, \mathfrak{b}" near end, dashed]     &&  
\end{tikzcd}
\]
\caption{Equivalences for small spaces} 
\end{figure}
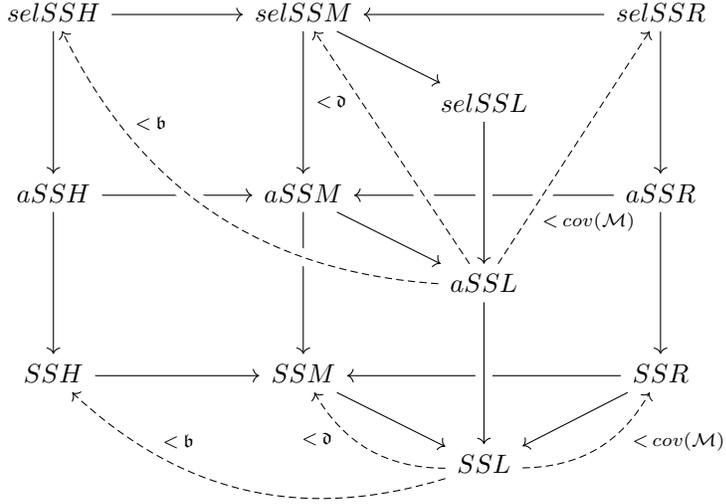

\begin{observation}
\label{T1SelSSMimpliesStronglysel(a)}
Recall that in $T_1$ spaces, finite sets are closed and discrete. Hence, if $X$ a $T_1$ selectively strongly star-Menger space, then $X$ is strongly selectively $(a)$.
\end{observation}

\begin{observation}
\label{selaimpliesabsstarCD}
If $X$ is selectively $(a)$ then $X$ is absolutely star-$\sigma$-cd, 
\end{observation}

By Corollary \ref{trivilind} (2) and (3) and Observations \ref{T1SelSSMimpliesStronglysel(a)} and \ref{selaimpliesabsstarCD} we get the following

\begin{corollary}
For $T_1$ spaces the following holds:
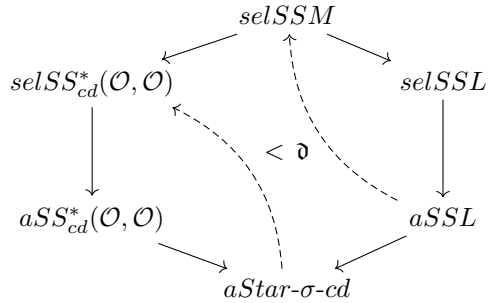
\begin{figure}[h!]
\[
\begin{tikzcd}[row sep=1em, column sep = 1em]
& selSSM   \arrow[dr]  \arrow[dl] & \\
selSS^*_{cd}(\mathcal{O},\mathcal{O}) \arrow[dd] & & selSSL     \\
&  < \, \mathfrak{d}  & \\
aSS^*_{cd}(\mathcal{O},\mathcal{O}) \arrow[dr] & & aSSL   \arrow[uu,<-,crossing over]  \arrow[uuul, bend left, dashed]  \arrow[dl]  \\
& aStar\text{-}\sigma\text{-}cd  \arrow[uuul, bend right,  dashed]  & 
\end{tikzcd}
\]
\caption{The dashed arrows hold if a space has size less than $\mathfrak{d}$.} 
\end{figure}
\end{corollary}

Other applications of Theorem \ref{general01} involve the following properties (see also \cite{M}):

\begin{itemize}[nosep]
    \item  \cite{YKS01} star-$K$-Menger.
    \item   \cite{YKS07}
    $\sigma$-starcompact. 
    \item   \cite{YKS08}  $\mathcal{L}$-starcompact.
   \item   \cite{YKS09}  $\mathcal{K}$-starcompact
\end{itemize}

\begin{remark}
Note that with the notation of definition \ref{starP} and the notion of star-$\mathbb{P}$ spaces, we have the following well-known implications:
\begin{center}
    star separable $\Leftrightarrow$ star countable (SSL) $\Rightarrow$ star compact $\Rightarrow$ star $\sigma$-compact $\Rightarrow$ star Lindel\"{o}f
\end{center}
Since the properties of the kernels are closed under finite unions, then we have the following
\end{remark}

\begin{proposition}
Let $X$ be a $\sigma$-starcompact space of size less than $\mathfrak{d}$. Then $X$ is star-$K$-Menger.
\end{proposition}

\begin{proposition}
Let $X$ be a $\mathcal{L}$-starcompact space of size less than $\mathfrak{d}$. Then $X$ is star-$\mathcal{L}$-Menger.
\end{proposition}

\begin{proposition}
Let $X$ be a $K$-starcompact space of size less than $\mathfrak{d}$. Then $X$ is star-$K$-Menger.
\end{proposition}

\section{The extent of absolutely star selection principles}

In the first part of this section we give a bound for the extent of absolutely star selection principles on separable spaces and, in the second part of the section, we provide some results regarding to selectively $(a)$-type properties on small spaces with countable extent. \\

\noindent Sakai showed in \cite{SM} that if $X$ is a strongly star-Menger space and $Y$ is a closed discrete subset of $X$, then $|Y| < cof(Fin(Y)^\mathbb{N})$ (in particular, $|Y| < \mathfrak{c}$). In addition, Caserta, Di Maio and Ko\v{c}inac pointed out in \cite{CDK} that if $X$ is a separable selectively-$(a)$ space and $Y \subseteq X$ is closed and discrete, then $|Y| < \mathfrak{c}$ (see \cite{dSSG2014} for the general case). Since, absolutely strongly star-Menger spaces are selectively-$(a)$, by this result we get that if $X$ is a separable absolutely strongly star-Menger space, then all its closed and discrete subsets are of size less than the continuum. Corollary \ref{SeparableaSSMsmallextent}, shows that in fact, such subsets are of size less than the dominating number $\mathfrak{d}$.\\

We will use the following definition in the proof of Theorem \ref{SeparableaNSMsmallextent} but will look in detail at it on the last section.

\begin{definition}
We say that a space $X$ is absolutely neighbourhood star-Menger if for each sequence $\{\mathcal{U}_n:n\in\omega\}$ of open covers and each dense subset $D$ of $X$, there exists a sequence $\{F_n:n\in\omega\}$ of finite subsets of $D$ such that for any open sets $O_n$ with $F_n\subseteq O_n$, $n\in\omega$, $\{St(O_n,\mathcal{U}_n):n\in\omega\}$ is an open cover of $X$.
\end{definition}

\begin{theorem}
\label{SeparableaNSMsmallextent}
Let $X$ be a separable absolutely neighbourhood star-Menger space. If $Y$ is a closed and discrete subset of $X$, then $|Y| < \mathfrak{d}$.
\end{theorem}

\begin{proof}
Assume $X$ is a separable absolutely neighbourhood star-Menger space. Then the set of isolated points is countable (otherwise it cannot be separable)  and it has to be a subset of every dense set. Hence, without loss of generality, we can assume $X$ has no isolated points. Now, if $Y$ is a closed and discrete subset of $X$, given that $X$ is separable and it has no isolated points, there is $E =\{e_n:n\in\omega\} \in [X \smallsetminus Y]^{\omega}$ so that $cl_X(E)= X$.

\noindent Now, let us assume that $|Y| \ge \mathfrak{d}$. Let $\{f_\alpha:\alpha < \mathfrak{d}\} \subseteq \omega ^\omega$ be a dominating family so that for each $g \in \omega ^\omega$ there is $\alpha < \mathfrak{d}$ such that $g \le f_\alpha$. For each $\alpha< \mathfrak{d}$, choose distinct points $p_{\alpha} \in Y$ and let $P = \{p_{\alpha}:\alpha < \mathfrak{d}\}$. For each $\alpha < \mathfrak{d}$, each $n \in \omega$ and each $i \le f_\alpha(n)$, define $O_n(p_{\alpha})$ and $V_n^{i \, \alpha}$ open sets so that 
\begin{enumerate}[nosep]
    \item $O_n(p_{\alpha})$ is an open neighbourhood of $p_{\alpha}$, 
    \item $O_n(p_{\alpha}) \cap Y = \{p_{\alpha}\}$,
    \item $V_n^{i \, \alpha}$ is an open neighbourhood of $e_i$, and
    \item $O_n(p_{\alpha}) \cap \bigcup _{i \le f_\alpha(n)} V_n^{i \, \alpha} = \emptyset$.
\end{enumerate}

\noindent For each $n\in \omega$ define $\mathcal{U}_n = \{O_n(p_{\alpha}): \alpha < \mathfrak{d}\} \cup \{X \smallsetminus P\}$. Observe that for each $n\in \omega$, $\mathcal{U}_n$ is an open cover of $X$.\\
We will show that the sequence $\{\mathcal{U}_n:n\in\omega\}$ and the dense set $E$, witness $X$ is not absolutely neighbourhood star-Menger.
Let $\{F_n:n\in\omega\}$ be any sequence of finite subsets of $E$. For each $n<\omega$, let $g(n) = min\{m:F_n \subseteq \{e_0, e_1, \ldots, e_m\}\}$. Thus, there is $\alpha < \mathfrak{d}$ such that for each $n\in \omega$, $g(n) \leq f_\alpha(n)$. 
\noindent If for each $n\in \omega$ we let $W_n = \bigcup_{i \le f_\alpha(n)}V_n^{i \, \alpha}$, then the sequence $(W_n:n\in \omega)$ satisfies that for each $n$, $F_n \subseteq W_n$. It only remains to show that $p_{\alpha} \notin \bigcup \{St(W_n,\mathcal{U}_n):n\in \omega\}$. Suppose the opposite, then there is $n \in \omega$ such that $p_{\alpha} \in St(W_n,\mathcal{U}_n)$. Since $O_n(p_{\alpha})$ is the only element of $\mathcal{U}_n$ that contains $p_{\alpha}$, then $O_n(p_{\alpha}) \cap W_n \ne \emptyset$. Then there is $i  \le f_\alpha(n)$ such that $ O_n(p_{\alpha}) \cap V_n^i \neq \emptyset$, which is a contradiction. Hence, $X$ is not absolutely neighbourhood star-Menger.
\end{proof}

Since every absolutely strongly star-Menger space is absolutely neighbourhood star-Menger, the following holds.

\begin{corollary}
\label{SeparableaSSMsmallextent}
Let $X$ be a separable absolutely strongly star-Menger space. If $Y$ is a closed and discrete subset of $X$, then $|Y| < \mathfrak{d}$.
\end{corollary}

\noindent The proof of Theorem \ref{SeparableaNSMsmallextent} is a modification of the proof of Proposition 2 in \cite{BM}, where Bonanzinga and Matveev show that if an almost disjoint family $\mathcal{A}$ is so that the Mr\'owka-Isbell space $\Psi(A)$ is strongly star-Menger, then $|\mathcal{A}| < \mathfrak{d}$. \\

\noindent {\bf Question:} Are aSSM and aNSM equivalent?
If not, can we find a normal (Tychonoff) counterexample?\\

Defining similarly the absolutely neighbourhood star-Hurewicz property (see last section), the Hurewicz case of Theorem \ref{SeparableaNSMsmallextent} can also be proved:

\begin{theorem}
\label{SeparableaNSHsmallextent}
Let $X$ be a separable absolutely neighbourhood star-Hurewicz space. If $Y$ is a closed and discrete subset of $X$, then $|Y| < \mathfrak{b}$.
\end{theorem}

Since every absolutely strongly star-Hurewicz space is absolutely neighbourhood star-Hurewicz, the following holds.

\begin{corollary}
\label{SeparableaSSHsmallextent}
Let $X$ be a separable absolutely strongly star-Hurewicz space. If $Y$ is a closed and discrete subset of $X$, then $|Y| < \mathfrak{b}$.
\end{corollary}

Similarly, we are interested in the question whether the properties
 $aSSH$ and $aNSH$ are equivalent.\\

The following theorem show the interplay between some selective versions of property $(a)$ and the selective and absolute versions of the star selection principles on small spaces with countable extent.

\begin{theorem}\label{diagramextent}
The following diagram holds for any space $X$. The dashed arrows hold if $X$ has countable extent and size less than the respective small cardinal invariant.
\end{theorem}

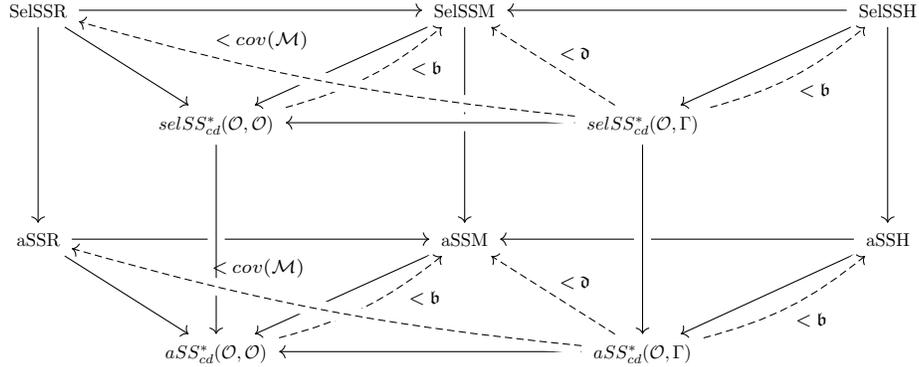
{\centering
\begin{tikzcd}[row sep=3em, column sep=2.6em]
 \scalebox{0.7}{SelSSR} \arrow[rrr] \arrow[dd] \arrow[dr] & & & \scalebox{0.7}{SelSSM} \arrow[rrr,<-] \arrow[dd] \arrow[dll] & & & \scalebox{0.7}{SelSSH} \arrow[dd] \arrow[dll] \\ 
 & \scalebox{0.7}{$selSS^*_{cd}(\mathcal{O},\mathcal{O})$} \arrow[rrr,<-,crossing over] \arrow[urr, crossing over, bend right=12, swap, "< \, \mathfrak{b}" near end, dashed]
 & & & \scalebox{0.7}{$selSS^*_{cd}(\mathcal{O},\Gamma)$} \arrow[ullll, crossing over, bend left=5, swap, "< \, cov(\mathcal{M})" near end, dashed] \arrow[ul, crossing over, "< \, \mathfrak{d}", swap, dashed] \arrow[urr, crossing over, bend right=12, swap, "< \, \mathfrak{b}" , dashed] & & \\ 
 \scalebox{0.7}{aSSR} \arrow[rrr] \arrow[dr] & & & \scalebox{0.7}{aSSM} \arrow[rrr,<-] \arrow[dll] & & & \scalebox{0.7}{aSSH} \arrow[dll] \\ 
 & \scalebox{0.7}{$aSS^*_{cd}(\mathcal{O},\mathcal{O})$} \arrow[uu,<-,crossing over] \arrow[rrr,<-,crossing over] \arrow[urr, crossing over, bend right=12, swap, "< \, \mathfrak{b}" near end, dashed] & & & \scalebox{0.7}{$aSS^*_{cd}(\mathcal{O},\Gamma)$} \arrow[uu,<-,crossing over] \arrow[ullll, crossing over, bend left=5, swap, "< \, cov(\mathcal{M})" near end, dashed] \arrow[ul, crossing over, "< \, \mathfrak{d}", swap, dashed] \arrow[urr, crossing over, bend right=12, swap, "< \, \mathfrak{b}" , dashed] & & \\ 
\end{tikzcd}\par
}

Observe the solid arrows follow from definitions. The ideas in the proofs of the eight dashed arrows are similar. So, we will only prove some of them in Lemmas \ref{gbh}, \ref{gdm}, \ref{obm} and \ref{gcovr}.  

\begin{lemma}\label{gbh}
If $X$ is $selSS^*_{cd}(\mathcal{O},\Gamma)$ with countable extent and $|X|<\mathfrak{b}$, then $X$ is $selSSH$.
\end{lemma}

\begin{proof}
Assume $X = \{x_\alpha: \alpha < \kappa\}$ ($\kappa<\mathfrak{b}$) is $selSS^*_{cd}(\mathcal{O},\Gamma)$ with countable extent. Let $(\mathcal{U}_n: n\in \omega)$ be any sequence of open covers of $X$ and let $(D_n:n\in \omega)$ be any sequence of dense subsets of $X$. Since $X$ is $selSS^*_{cd}(\mathcal{O},\Gamma)$, for each $n \in \omega$, we can take $C_n\subseteq D_n$ closed discrete such that $\{St(C_n,\mathcal{U}_n):n\in\omega\}$ is a $\gamma$-cover of $X$. We can assume each $C_n$ to be countable. Thus, for each $n\in\omega$, we enumerate $C_n=\{c_i^n:i\in\omega\}$.\\
For each $\alpha < \kappa$ we define $f_\alpha\in \omega^\omega$ such that for $n\in\omega$, $f_\alpha(n)=min\{m\in\omega:x_\alpha \in St(c^n_m,\mathcal{U}_n)\}$ if $x_\alpha \in St(C_n,\mathcal{U}_n)$ and $f_\alpha(n)=0$ otherwise. Since $|\{f_\alpha:\alpha < \kappa\}|<\mathfrak{b}$, there exists a function $g\in\omega^\omega$ such that for each $\alpha<\kappa$, $f_\alpha\leq^* g$. We define, for each $n\in\omega$, $F_n=\{c_i^n:i\leq g(n)\}$. It follows that the collection $\{St(F_n,\mathcal{U}_n):n\in \omega\}$ is a $\gamma$-cover of $X$. Indeed, let $x_\alpha \in X$. Then, there is $n_0\in\omega$ so that for every $n\geq n_0$, $x_\alpha\in St(C_n,\mathcal{U}_n)$. Since $f_\alpha \leq ^* g$, there is $n_1\in\omega$ such that for every $n\geq n_1$, $f_\alpha(n)\leq g(n)$. Put $m= max\{n_0, n_1\}$. Thus, if $n\geq m$, then $x_\alpha \in St(c^n_{f_\alpha(n)},\mathcal{U}_n)\subseteq St(F_n, \mathcal{U}_n)$. Therefore, the collection $\{St(F_n, \mathcal{U}_n):n\in\omega\}$ is a $\gamma$-cover of $X$. Thus, $X$ is selectively strongly star-Hurewicz.
\end{proof}

\begin{lemma}\label{gdm}
If $X$ is $selSS^*_{cd}(\mathcal{O},\Gamma)$ with countable extent and $|X|<\mathfrak{d}$, then $X$ is $selSSM$.
\end{lemma}

\begin{proof}
By mimicking the first part of the proof of Lemma \ref{gbh}, we get a family of functions $\{f_\alpha:\alpha < \kappa\}$ with $\kappa < \mathfrak{d}$.
Then, there exists a function $g\in\omega^\omega$ such that for each $\alpha<\kappa$, $g \nleq^* f_\alpha$. We define, for each $n\in\omega$, $F_n=\{c_i^n:i\leq g(n)\}$. Let us show that the collection $\{St(F_n,\mathcal{U}_n):n\in \omega\}$ is an open cover of $X$. Let $x_\alpha \in X$. Then, there is $n_0\in\omega$ so that for every $n\geq n_0$, $x_\alpha\in St(C_n,\mathcal{U}_n)$. Since $g \nleq ^* f_\alpha$, there is $n_1\geq n_0$ such that $g(n_1) > f_\alpha(n_1)$. Thus, we obtain that $x_\alpha \in St(c^{n_1}_{f_\alpha(n_1)},\mathcal{U}_{n_1})\subseteq St(F_{n_1}, \mathcal{U}_{n_1})$. Hence, the collection $\{St(F_n, \mathcal{U}_n):n\in\omega\}$ is an open cover of $X$. Thus, $X$ is selectively strongly star-Menger.
\end{proof}

\begin{lemma}\label{obm}
If $X$ is $selSS^*_{cd}(\mathcal{O},\mathcal{O})$ with countable extent and $|X|<\mathfrak{b}$, then $X$ is $selSSM$.
\end{lemma}

\begin{proof}
List $X = \{x_\alpha: \alpha < \kappa\}$ ($\kappa<\mathfrak{b}$). Let $(\mathcal{U}_n: n\in \omega)$ be any sequence of open covers of $X$ and let $(D_n:n\in \omega)$ be any sequence of dense subsets of $X$. For each $n \in \omega$, take $C_n\subseteq D_n$ closed discrete such that for any $k\in\omega$, $\{St(C_n,\mathcal{U}_n):n\geq k\}$ is an open cover of $X$. 
For each $\alpha < \kappa$ we define $f_\alpha$ as in the proof of Lemma \ref{gbh}. Since $|\{f_\alpha:\alpha < \kappa\}|<\mathfrak{b}$, there exists a function $g\in\omega^\omega$ such that for each $\alpha<\kappa$, $f_\alpha\leq^* g$. We define, for each $n\in\omega$, $F_n=\{c_i^n:i\leq g(n)\}$. It follows that the collection $\{St(F_n,\mathcal{U}_n):n\in \omega\}$ is an open cover of $X$. Indeed, let $x_\alpha \in X$. Then, there is $n_0\in\omega$ so that for every $n\geq n_0$, $f_\alpha(n)\leq g(n)$. In addition, for such $n_0$, there exists $n_1\geq n_0$ such that $x_\alpha\in St(C_{n_1}, \mathcal{U}_{n_1})$. In particular, we get that $x_\alpha \in St(c^{n_1}_{f_\alpha(n_1)},\mathcal{U}_{n_1})\subseteq St(F_{n_1}, \mathcal{U}_{n_1})$. Therefore, the collection $\{St(F_n, \mathcal{U}_n):n\in\omega\}$ is an open cover of $X$. Hence, $X$ is selectively strongly star-Menger.
\end{proof}

\begin{lemma}\label{gcovr}
If $X$ is $selSS^*_{cd}(\mathcal{O},\Gamma)$ with countable extent and $|X|<cov(\mathcal{M})$, then $X$ is $selSSR$.
\end{lemma}

\begin{proof}
In the same way as in Lemma \ref{gbh}, we can define a family of functions $\{f_\alpha:\alpha < \kappa\}$ with $\kappa < cov(\mathcal{M})$. Then,
there exists a function $g\in\omega^\omega$ such that for each $\alpha<\kappa$, $|\{n\in\omega: f_\alpha(n)=g(n)\}|=\omega$. Let us show that the collection $\{St(c_{g(n)}^n,\mathcal{U}_n):n\in \omega\}$ is an open cover of $X$. Let $x_\alpha \in X$. Then, there is $n_0\in\omega$ so that for every $n\geq n_0$, $x_\alpha\in St(C_n,\mathcal{U}_n)$. Since $g(n)=f_\alpha(n)$ for infinitely many $n$, we can take $n_1\geq n_0$ such that $g(n_1)=f_\alpha(n_1)$. Then, $x_\alpha \in St(c^{n_1}_{f_\alpha(n_1)},\mathcal{U}_{n_1}) = St(c^{n_1}_{g(n_1)},\mathcal{U}_{n_1})$. Therefore, the collection $\{St(c_{g(n)}^n,\mathcal{U}_n):n\in \omega\}$ is an open cover of $X$. Hence, $X$ is selectively strongly star-Rothberger.
\end{proof}

We mention a couple of immediate consequences of Theorem \ref{diagramextent}:

\begin{corollary}
In spaces of size less than $\mathfrak{b}$ with countable extent, the following properties are equivalent:
\begin{enumerate}
    \item selectively strongly star-Menger;
    \item strongly selectively (a)
\end{enumerate}
\end{corollary}

\begin{corollary}
In spaces of size less than $\mathfrak{b}$ with countable extent, the following properties are equivalent:
\begin{enumerate}
    \item absolutely strongly star-Menger;
    \item selectively (a)
\end{enumerate}
\end{corollary}

\begin{remark}
The hypothesis of the countable extent in Theorem \ref{diagramextent} is necessary. Assuming $\omega_1<\mathfrak{b}$ we have that the discrete space of size $\omega_1$ is $selSS^*_{cd}(\mathcal{O},\Gamma)$ but is not selSSM. 
\end{remark}

\section{Absolute and selective versions of the neighbourhood star selection principles}

The selective and absolute versions of the properties $SSM$, $SSR$ and $SSH$ were studied in this paper. A different sort of these star selection principles which is closely related to the properties studied here is the neighbourhood version of the star selection principles. The definitions of these neighbourhood star selection principles were given in \cite{K2} (with different name) and studied in \cite{BCKM}. In this final section, using similar ideas as before, we introduce the absolute and selective versions of the neighbourhood star selection principles. Some further investigations on these kind of versions (which we have just begun to study) may be interesting.

We start by mentioning the definition of the absolute and selective versions of the neighbourhood star-Lindel\"of property (the neighbourhood star-Lindel\"of property was introduced in \cite{BCKM} and later studied by Song in \cite{YKS05} and \cite{YKS06}). 

\begin{definition}
A space $X$ is absolutely neighbourhood star-Lindel\"{o}f ($aNSL$) if for any open cover $\mathcal{U}$ of $X$ and any dense subset $D$ of $X$, there is a countable set $C\subseteq D$ such that for any open set $O$ with $C\subseteq O$,  $St(O,\mathcal{U})=X$.
\end{definition}

\begin{definition}
A space $X$ is selectively neighbourhood star-Lindel\"of ($selNSL$) if for any open cover $\mathcal{U}$ of $X$ and any sequence $\{D_n:n\in\omega\}$ of dense sets of $X$, there are finite sets $F_n\subseteq D_n$, $n\in\omega$, such that for any open sets $O_n$ in $X$ with $F_n\subseteq O_n$, $n\in\omega$, $\{St(O_n,\mathcal{U}):n\in\omega\}$ is an open cover of $X$.
\end{definition}

Following the same notation and terminology of this article, we introduce general forms of two selection hypothesis which allows us to define the absolute and selective versions of the neighbourhood star selection properties.

\begin{definition}
Given a space $X$, the following selection hypothesis are defined:\\
\newline
$\mathbf{absolutely \; NSS^*_{\mathcal{K}}(\mathcal{A},\mathcal{B})}$: For each sequence $\{\mathcal{U}_n:n\in\omega\}\subseteq\mathcal{A}$ and each dense subset $D$ of $X$, there exists a sequence $\{K_n:n\in\omega\}\subseteq\mathcal{K}$ with $K_n\subseteq D$, $n\in\omega$, such that for any open sets $O_n$ with $K_n\subseteq O_n$, $n\in\omega$, $\{St(O_n,\mathcal{U}_n):n\in\omega\}\in\mathcal{B}$.\\
\newline
$\mathbf{selectively \; NSS^*_{\mathcal{K}}(\mathcal{A},\mathcal{B})}$: For each sequence $\{\mathcal{U}_n:n\in\omega\}\subseteq\mathcal{A}$ and each sequence $\{D_n:n\in\omega\}$ of dense sets of $X$, there exists a sequence $\{K_n:n\in\omega\}\subseteq\mathcal{K}$ with $K_n\subseteq D_n$, $n\in\omega$, such that for any open sets $O_n$ with $K_n\subseteq O_n$, $n\in\omega$, $\{St(O_n,\mathcal{U}_n):n\in\omega\}\in\mathcal{B}$. 
\end{definition}

For shortness, we write $\mathbf{aNSS^*_{\mathcal{K}}(\mathcal{A},\mathcal{B})}$ instead of $\mathbf{absolutely \; NSS^*_{\mathcal{K}}(\mathcal{A},\mathcal{B})}$ and, $\mathbf{selNSS^*_{\mathcal{K}}(\mathcal{A},\mathcal{B})}$ instead of $\mathbf{selectively \; NSS^*_{\mathcal{K}}(\mathcal{A},\mathcal{B})}$. Therefore, we introduce the following new properties:

\begin{description}
\item[aNSM]: $aNSS^*_{fin}(\mathcal{O},\mathcal{O})$ defines the absolutely neighbourhood star-Menger property;
\item[aNSR]: $aNSS^*_1(\mathcal{O},\mathcal{O})$ defines the absolutely neighbourhood star-Rothberger property;
\item[aNSH]: $aNSS^*_{fin}(\mathcal{O},\Gamma)$ defines the absolutely neighbourhood star-Hurewicz property;
\item[selNSM]: $selNSS^*_{fin}(\mathcal{O},\mathcal{O})$ defines the selectively neighbourhood star-Menger property;
\item[selNSR]: $selNSS^*_1(\mathcal{O},\mathcal{O})$ defines the selectively neighbourhood star-Rothberger property;
\item[selNSH]: $selNSS^*_{fin}(\mathcal{O},\Gamma)$ defines the selectively neighbourhood star-Hurewicz property.
\end{description}

\newpage 

Obvious implications among these absolute and selective versions provide the following diagram:\\

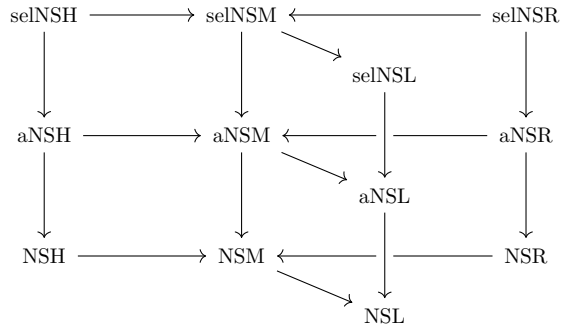
{\centering
\begin{tikzcd}[row sep=1em, column sep = 2em]
\scalebox{0.8}{selNSH} \arrow[rr] \arrow[dd] && \scalebox{0.8}{selNSM}  \arrow[rr,<-]  \arrow[dr] \arrow[dd] &&
  \scalebox{0.8}{selNSR} \arrow[dd ] \\
& & & \scalebox{0.8}{selNSL} & \\
\scalebox{0.8}{aNSH} \arrow[rr] \arrow[dd] && \scalebox{0.8}{aNSM} \arrow[rr,<-] \arrow[dr]  \arrow[dd]  && \scalebox{0.8}{aNSR} \arrow[dd]  \\
& & & \scalebox{0.8}{aNSL}  \arrow[uu,<-,crossing over] & \\
\scalebox{0.8}{NSH} \arrow[rr] && \scalebox{0.8}{NSM} \arrow[rr,<-] \arrow[dr]  && \scalebox{0.8}{NSR}  \\
& & & \scalebox{0.8}{NSL}  \arrow[uu,<-,crossing over] &
\end{tikzcd}\par
}

\vspace{1cm}
By including the selective and absolute versions of the strongly star principles, we get the following general diagram that involves all selective and absolute versions considered so far:\\

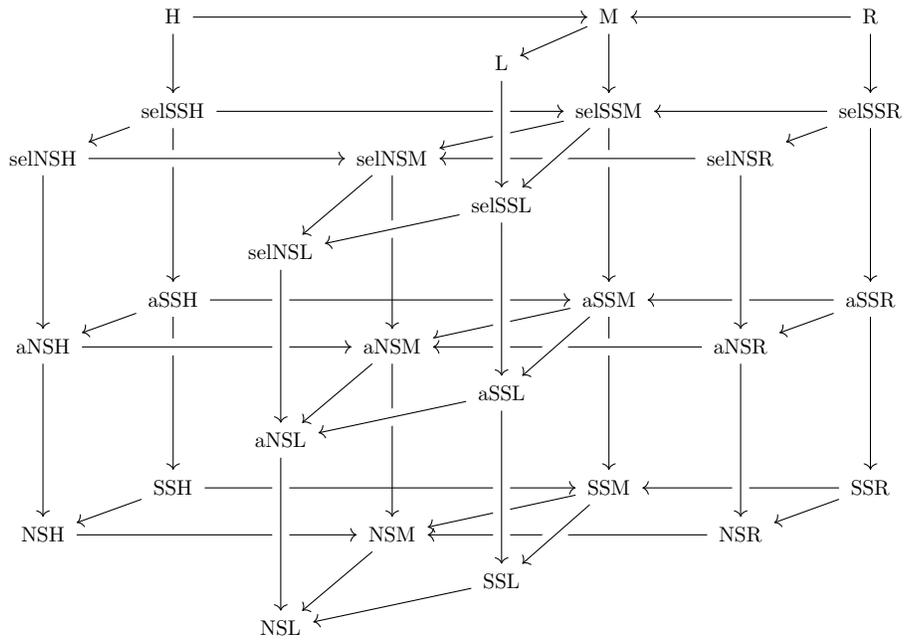
{\centering
\begin{tikzcd}[row sep=.5em, column sep=.8em]
 & & \scalebox{0.8}{H} \arrow[rrrr] \arrow[dd] & & & & \scalebox{0.8}{M} \arrow[rrrr,<-] \arrow[dd] \arrow[dl] & & & & \scalebox{0.8}{R} \arrow[dd] \\ 
 & & & & & \scalebox{0.8}{L} & & & & & \\ 
 & & \scalebox{0.8}{selSSH} \arrow[rrrr] \arrow[dddd] \arrow[dll] & & & & \scalebox{0.8}{selSSM} \arrow[rrrr,<-] \arrow[dddd] \arrow[dll] & & & & \scalebox{0.8}{selSSR} \arrow[dddd] \arrow[dll] \\ 
 \scalebox{0.8}{selNSH} \arrow[rrrr, crossing over] \arrow[dddd] & & & & \scalebox{0.8}{selNSM} \arrow[rrrr,<-,crossing over] \arrow[dddd,crossing over] \arrow[ddl] & & & & \scalebox{0.8}{selNSR} & & \\ 
 & & & & & \scalebox{0.8}{selSSL} \arrow[uuu,<-,crossing over] \arrow[uur,<-,crossing over] \arrow[dll,crossing over] & & & & & \\ 
 & & & \scalebox{0.8}{selNSL} & & & & & & & \\ 
 & & \scalebox{0.8}{aSSH} \arrow[rrrr] \arrow[dddd] \arrow[dll] & & & & \scalebox{0.8}{aSSM} \arrow[rrrr,<-] \arrow[dddd] \arrow[dll] & & & & \scalebox{0.8}{aSSR} \arrow[dddd] \arrow[dll] \\ 
 \scalebox{0.8}{aNSH} \arrow[rrrr,crossing over] \arrow[dddd] & & & & \scalebox{0.8}{aNSM} \arrow[rrrr,<-,crossing over] \arrow[dddd,crossing over] \arrow[ddl] & & & & \scalebox{0.8}{aNSR} \arrow[uuuu,<-,crossing over] & & \\ 
 & & & & & \scalebox{0.8}{aSSL} \arrow[uuuu,<-,crossing over] \arrow[uur,<-,crossing over] & & & & & \\ 
 & & & \scalebox{0.8}{aNSL} \arrow[uuuu,<-,crossing over] \arrow[urr,<-,crossing over] & & & & & & & \\ 
 & & \scalebox{0.8}{SSH} \arrow[rrrr] \arrow[dll] & & & & \scalebox{0.8}{SSM} \arrow[rrrr,<-] \arrow[dll] & & & & \scalebox{0.8}{SSR} \arrow[dll] \\ 
 \scalebox{0.8}{NSH} \arrow[rrrr] & & & & \scalebox{0.8}{NSM} \arrow[rrrr,<-] & & & & \scalebox{0.8}{NSR} \arrow[uuuu,<-,crossing over] & & \\ 
 & & & & & \scalebox{0.8}{SSL} \arrow[uuuu,<-,crossing over] \arrow[uur,<-,crossing over] & & & & & \\ 
 & & & \scalebox{0.8}{NSL} \arrow[uuuu,<-,crossing over] \arrow[uur,<-] \arrow[urr,<-] & & & & & & & \\ 
\end{tikzcd}\par
}

\section*{Acknowledgements}

The second author was partly supported for this research by the Consejo Nacional de Ciencia y Tecnolog\'ia CONACYT, M\'exico, Scholarship 411689.\\

\small
\baselineskip=5pt


\textsc{Department of Mathematics and Statistics, York University, 4700 Keele St. Toronto, ON M3J 1P3 Canada}\par\nopagebreak
\textit{Email address}: J. Casas-de la Rosa: \texttt{olimpico.25@hotmail.com}\\ S. A. Garcia-Balan: \texttt{sgarciab@yorku.ca}

\end{document}